\documentclass[12pt]{amsart}

\usepackage{amsmath}
\usepackage{amssymb}
\usepackage{amsfonts}
\usepackage{amsthm}
\usepackage{enumerate}
\usepackage{hyperref}
\usepackage{paralist}
\usepackage{color}
\usepackage{thmtools}
\usepackage{changepage}
\usepackage{cleveref}
\usepackage{graphicx}
\numberwithin{equation}{section}

\newcommand{\NN}{\mathbb{N}}
\newcommand{\set}[1]{\{#1\}}
\newcommand{\with}{\,|\,}
\newcommand{\ZZ}{\mathbb{Z}}

\DeclareMathOperator{\sdepth}{sdepth}
\DeclareMathOperator{\hdepth}{hdepth}
\DeclareMathOperator{\depth}{depth}

\DeclareMathOperator{\ind}{ind}
\newcommand{\wnu}{\widetilde{\nu}}
\newcommand{\walpha}{\widetilde{\alpha}}
\newcommand{\wpsi}{\widetilde{\psi}}
\theoremstyle{plain}
\newtheorem{thm}{Theorem}[section]

\newtheorem{prop}[thm]{Proposition}
\newtheorem{lem}[thm]{Lemma}

\theoremstyle{definition}
\newtheorem{dfn}[thm]{Definition}

\newtheorem{exmp}[thm]{Example}

\newtheorem{rem}[thm]{Remark}

\newtheorem{dfns-rems}[thm]{Definitions and Remarks}
\newtheorem{notas-rems}[thm]{Notations and Remarks}
\newtheorem{exmps-rems}[thm]{Examples and Remarks}

\newlength\Thmindent
\begin{document}

\title[The Stanley depth in the Koszul Complex]{The Stanley depth in the upper half of the Koszul Complex}

\author{Lukas Katth\"an}

\address{Lukas Katth\"an, Universit\"at Osnabr\"uck, FB Mathematik/Informatik, 49069 Osnabr\"uck, Germany}
\email{lkatthaen@uos.de}

\author{Richard Sieg}

\address{Richard Sieg, Universit\"at Osnabr\"uck, FB Mathematik/Informatik, 49069 Osnabr\"uck, Germany}

\email{richard.sieg@uos.de}

\begin{abstract}
Let $R = K[X_1, \dotsc, X_n]$ be a polynomial ring over some field $K$.
In this paper, we prove that the $k$-th syzygy module of the residue class field $K$ of $R$ has Stanley depth $n-1$ for $\lfloor n/2 \rfloor \leq k < n$, as it had been conjectured by Bruns et.\,al.\, in 2010.
In particular, this gives the Stanley depth for a whole family of modules whose graded components have dimension greater than $1$. So far, the Stanley depth is known only for a few examples of this type. Our proof consists in a close analysis of a matching in the Boolean algebra.
\end{abstract}



 \thanks{Both authors were partially supported by the German Research Council DFG-GRK~1916. }

\maketitle

\section{Introduction} \label{sec1}
Stanley decompositions and Stanley depth form an important and investigated topic in combinatorial commutative algebra. These decompositions split a module into a direct sum of graded vector spaces of the form $Sm$, where $S=K[X_{i_1},\dots,X_{i_d}]$ is a subalgebra of the polynomial ring and $m$ a homogeneous element. They were introduced by Stanley in \cite{sta82} and he conjectured that the maximal depth of all possible decompositions of a module - the \emph{Stanley depth} - is greater than or equal to the usual depth of the module. Nowadays this question runs under the name \emph{Stanley conjecture} and is still open.

Bruns et.\,al. introduced a weaker notion of Stanley decompositions in \cite{bku10}, namely \emph{Hilbert decompositions}. In contrast to the former they only depend on the Hilbert series of the module and are usually easier to compute. The analogue of the Stanley depth - the \emph{Hilbert depth} - gives a natural upper bound for the Stanley depth of every module and leads to a weakened formulation of the Stanley conjecture. 

Let $R = K[X_1, \dotsc, X_n]$ be a polynomial ring over some field $K$.
Let us denote by $M(n,k)$ the $k$-th syzygy module of the residue field $K$ of $R$, i.\,e.\,the $k$-th syzygy module in the Koszul complex.
It was shown in the named paper that the Hilbert depth in the upper half of the Koszul complex is $n-1$, where $n$ is the number of variables and  conjectured that the same is true for the Stanley depth. In this paper we prove this conjecture. Our main theorem is the following:
\begin{thm}\label{mthm1}
For all $n$ and $n>k\geq\lfloor\frac{n}{2}\rfloor$ one has $$\sdepth M(n,k)=n-1.$$
\end{thm}
If a module is \emph{finely graded}, i.\,e.\,every graded part has $K$-dimension at most $1$, it is rather easy to transform a Hilbert decomposition into a Stanley decomposition and thus also to compute the respective depth (see \cite[Proposition 2.8]{bku10}). However, this is not the case for the modules of our interest. In particular 
$$\dim M(n,k)_m=\binom{|m|-1}{k-1}$$ where $m$ is a multidegree and $|m|$ its total degree. Hence our theorem provides the Stanley depth for a whole family of modules with higher dimensions in the graded components. Up to now only a few examples of this type are known. To obtain the result we transform the Hilbert decomposition in \cite{bku10} into a Stanley decomposition by applying new combinatorial techniques. Especially we are interested in matchings in the Boolean algebra and their properties.

In Section~\ref{sec:shd} we review the definitions of Stanley and Hilbert decompositions and respective depth and their connections. The next section deals with a matching in the Boolean algebra and its properties, in particular a concrete formula for an injective map from bigger to smaller sets in the upper half of this poset. This part mainly relies on a paper by Aigner (see \cite{aig73}). In the last section we firstly review the Hilbert decomposition used for the proof of Bruns et.\,al.\,for the Hilbert depth in the upper half of the Koszul complex. We then introduce the notion of the \emph{index} of a subset $G$ of a given set $M$. It is the highest power of the matching restricted to the power set of $M$ for which $G$ is in its image. We argue that in order to prove the theorem, we have to show that the subsets of size $k$ with even index of every set $M$ have an order fulfilling a certain condition. As a final step it is shown that the \emph{squashed} order satisfies this condition. 

The methods and notions developed in this paper are not only important for the sake of the proof. They also give new interesting insights from a purely combinatorial point of view.

Note that we will not investigate the Stanley depth in the lower half of the Koszul complex. Already the Hilbert depth behaves quite irregular in this case as it was pointed out in the last sections of \cite{bku10}. Moreover, the techniques developed in this paper are rather special and cannot be applied in the lower half. 

For convenience we sometimes write $[n]$ for $\set{1,\dots,n}$. Furthermore we call a set with $l$ elements an $l$-set.
\section{Stanley and Hilbert Decompositions}\label{sec:shd}
We briefly review the basic concepts as given in \cite{bku10}. For a fuller treatment of Stanley decompositions see for example \cite{her13}.

We will denote by $R$ the polynomial ring $K[X_1,\dots,X_n]$ in $n$ variables over a field $K$. It is equipped with a multigrading over $\ZZ^n$, i.e.\,$\deg(X_i)=e_i$ where $e_i$ is the $i$-th unit vector. 

Let $M$ be a finitely generated graded $R$-module and $m\in M$ a homogeneous element. Furthermore let $Z\subseteq\{X_1,\dots,X_n\}$. The module $K[Z]m$ is called a \emph{Stanley space} of $M$ if $K[Z]m$ is a free $K[Z]$-submodule of $M$.

\begin{dfn}\label{def:sd}
Let $M$ be a finitely generated graded $R$-module. A \emph{Stanley decomposition} $$\mathcal{D}=(K[Z_i],m_i)_{i\in I}$$ is a finite decomposition of $M$ as a graded $K$-vector space $$M=\bigoplus_{i\in I}K[Z_i] m_i$$ where the $K[Z_i]m_i$'s are Stanley spaces of $M$.

This direct sum  forms an $R$-module and thereby has a depth. This allows us to define the \emph{Stanley depth} of $M$ as the maximal depth of all possible Stanley decompositions: $$\sdepth M:=\max\{\depth\mathcal{D}\mid \mathcal{D}\text{ is a Stanley decomposition of }M\}.$$
\end{dfn}
The Hilbert decomposition and depth are defined in a similar manner, but they only depend on the Hilbert series of $M$, which makes them easier to compute.
\begin{dfn}
Let $M$ and $R$ be as in Definition~\ref{def:sd}. A \emph{Hilbert decomposition} $$\mathcal{D}=(K[Z_i],s_i)_{i\in I}$$ is a finite family of modules $K[Z_i]$ and multidegrees $s_i\in\ZZ^n$, such that $$M\cong \bigoplus_{i\in I}K[Z_i](-s_i)$$ as a graded $K$-vector space.

Furthermore, the \emph{Hilbert depth} of $M$ is defined as the maximal depth of all possible Hilbert decompositions: $$\hdepth M:=\max\{\depth\mathcal{D}\mid \mathcal{D}\text{ is a Hilbert decomposition of }M\}.$$
\end{dfn}

Note that by definition we immediately have $$\hdepth M\geq\sdepth M$$ for any $R$-module $M$. Moreover, the Stanley and Hilbert depth can also be defined for a $\ZZ$-graded module, as it was down in \cite{bku10}.

\section{Matching in the Boolean Algebra}\label{mba}
We review the matching (i.\,e.\,an injection $f$ with $f(G)\subset G$) in the upper half of the Boolean algebra that will be used for the construction of the Hilbert and Stanley decomposition of the considered modules in the next section. This map is known in combinatorics for some time. Especially it was used to give a proof of \emph{Sperner's Theorem} (see for instance \cite{and87}).

The \emph{lexicographic mapping} $\psi$ is a (partially defined) injective map on the Boolean algebra which assigns to an $(l+1)$-set an $l$-sets in the following way. 
Firstly, we write down the Boolean algebra and sort each level lexicographically.
For each $(l+1)$-set $G\subseteq\{1,\dots,n\}$, $\psi(G)$ is the lexicographically smallest $l$-subset of $G$ that is not already in the image of $\psi$. If there exists no such subset, then $\psi$ is undefined.

While the definition makes clear that $\psi$ is injective, it is not very convenient to work with.
Aigner provided a concrete formula for the above matching (see \cite{aig73}).
Stanton and White gave a helpful pictorial interpretation of this formula in \cite{sw86} which will be discussed below. The following reformulation is motivated by this interpretation.

We associate to a set $G\subseteq\{1,\dots,n\}$ an $n$-dimensional incidence vector $$\chi_G(g)=\begin{cases}
1 & \text{if }g\in G; \\ -1& \text{if }g\notin G.
\end{cases}$$ 
Furthermore we set $\chi_G(0):=0$. Now we look at all elements of $G$, for which the function 
\begin{align*}
\rho_G(g)&:=\sum_{j=0}^{g}\chi_G(j),\quad g\in\{0,1,\dots,n\}
\intertext{is maximized:}
\alpha(G)&:=\max_{g\in G\cup\{0\}}\{\rho_{G}(g)\}\\
N(G)&:=\{g\in G\cup\{0\}\mid \rho_{G}(g)=\alpha(G)\},
\intertext{and take the smallest element}
\nu(G)&:=\min N(G).
\end{align*}
Then the map $\psi$ is defined by deleting the element $\nu(G)$: $$\psi(G):=G\setminus\{\nu(G)\}.$$ This map is defined if and only if $\nu(G)>0$ (equivalently if and only if $\alpha(G) > 0$).
As $\alpha(G) \geq \rho_G(n) = 2|G| - n$, this is indeed the case for all subsets $G \subseteq \{1,\dots,n\}$ with $|G|\geq\lceil\frac{n+1}{2}\rceil$.

As mentioned above Stanton and White provided a geometric interpretation of $\psi$. The vector $\chi_G$ can be seen as a lattice path, where a 1 means going one up to the right and $-1$ means going one down to the right. Then $\rho_{G}(g)$ determines the height in place $g$ and $N(G)$ consists of all global maxima. Therefore, $\psi$ ``flips'' the first global maximum (if it is above the $x$-axis), as illustrated in Figure~\ref{fig:psi}.
\begin{figure}[h]
\centering
\includegraphics{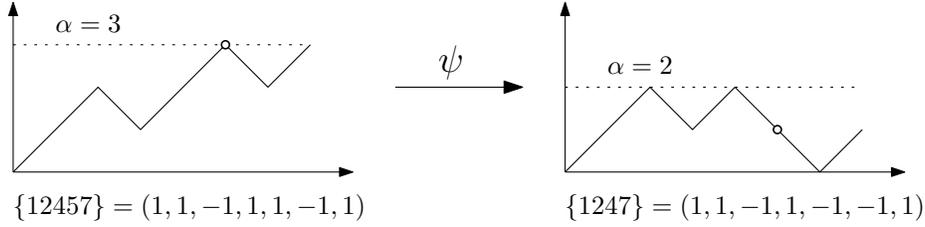}
\caption{Geometric interpretation of $\psi$}
\label{fig:psi}
\end{figure}

\subsection{The Inverse $\phi$}\label{phi}
Sometimes it is helpful to also consider the mapping in the other direction in the Boolean algebra called $\phi$. Again the map was given in explicit terms by Aigner in \cite{aig73}. We again look at the set $N(G)$ for which $\rho_{G}$ is maximized, but this time we take the maximal element: $$\mu(G)=\max N(G).$$ Then $$\phi(G)=G\cup \{\mu(G)+1\}.$$ So $\phi$ is defined if and only if $\mu(G)<n$. This is the case for all sets in the lower half of the Boolean algebra.

If the set is considered as a lattice path like above, $\phi$ ``flips'' the edge of the last global maximum up, i.\,e.\,changes the subsequent entry to a 1, see Figure~\ref{fig:phi}.
\begin{figure}[h]
\centering
\includegraphics{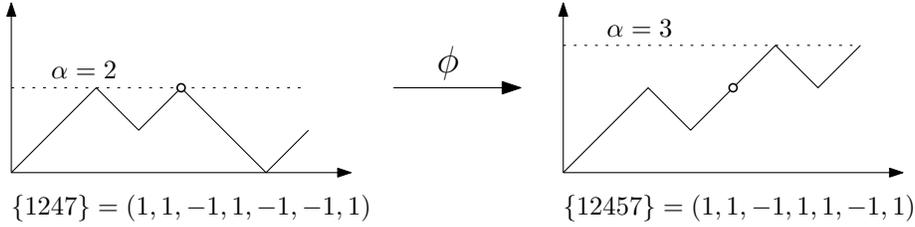}
\caption{Geometric interpretation of $\phi$}
\label{fig:phi}
\end{figure}

The following statement from \cite[Theorem 3]{aig73} will be used in later proofs:
\begin{prop}\label{psiphi}
A subset of $\{1,\dots,n\}$ is in the image of $\psi$ if and only if $\phi$ is defined on this set and vice versa. Furthermore $\psi$ and $\phi$ are inverse to each other on the respective domain. 
\end{prop}

\section{Hilbert and Stanley depth in the Koszul Complex}\label{sec:hsdko}
In order to prove our main result we need to review the arguments from \cite{bku10} which show that the Hilbert depth in the upper half of the Koszul complex is $n-1$.

Let $K$ be a field and $R=K[X_1,\dots,X_n]$.
Then $K=R/\mathfrak{m}, \mathfrak{m}=(X_1,\dots,X_n)$ and the \emph{Koszul complex} is the following minimal free resolution of $K$:
$$
0\to\bigwedge^n R ^n\stackrel{\partial}{\to}\bigwedge^{n-1}R^n\stackrel{\partial}{\to}\dots\stackrel{\partial}{\to}R^n\stackrel{\partial}{\to}R\to0
$$
where $\partial$ is the boundary operator 
\begin{equation}\label{eq:diff}
\partial(e_{i_1}\wedge\dots\wedge e_{i_k})=\sum_{j=1}^{k}(-1)^{j+1}X_j \,e_{i_1}\wedge\dots\wedge\widehat{e_{i_j}}\wedge\dots\wedge e_{i_k}.
\end{equation}

By $M(n,k)$ we denote the $k$-th syzygy module of $K$.
\newcommand{\Sc}{\mathcal{S}}
We recall the Hilbert decomposition of $M(n,k)$ constructed in \cite{bku10}.
Let 
\[ \Sc := \set{ S \subseteq [n] \with |S| = k + j, j \text{ even}}. \]
For $S \in \Sc$ we further set
\begin{equation}\label{eq:Zs}
	Z_S=\begin{cases}
		\{X_i\with i\in[n]\} &\text{if }S\text{ is not in the image of }\psi;\\
		\{X_i\with i\in[n]\setminus\set{s}\} & \text{if }S = \psi(S\cup\{s\}). 
	\end{cases}
\end{equation}
Then, as shown in the proof of \cite[Theorem 3.5]{bku10},
\begin{equation}
(K[Z_S],S)_{S \in \Sc}\label{eq:hdec}
\end{equation}
is a Hilbert decomposition of $M(n,k)$ for $n>k\geq\lfloor\frac{n}{2}\rfloor$.

Throughout the rest of the paper we let $n>k\geq\lfloor\frac{n}{2}\rfloor$ be fixed.

We will show that also the Stanley depth in the upper half of the Koszul complex is $n-1$ by turning the Hilbert decomposition \eqref{eq:hdec}
into a Stanley decomposition.
For this, we need to choose for every $S \in \Sc$ an element $m_S \in M(n,k)$ of multidegree $S$.

Note that $M(n,k)$ is the image of the $k$-th boundary map in the Koszul complex, hence it is generated by the elements $\partial (e_G)$, where $e_G := e_{i_1} \wedge \dotsb \wedge e_{i_k}$ with $G = \set{i_1, \dotsc, i_k} \subset [n]$.
Moreover, if $G \subseteq S$, then $X^{S \setminus G}\partial(e_G)$ has multidegree $S$, where $X^{S \setminus G}$ is the monomial of degree ${S \setminus G}$.
Hence, we essentially need to choose subsets $G(S) \subseteq S$ of cardinality $k$ for every $S \in \Sc$.
It turns out that the following choice works:
\begin{equation}\label{eq:choice}
G(S) := \psi^{|S|-k}(S) \qquad \text{ for } S \in \Sc.
\end{equation}

Thus we are going to prove the following theorem:
\begin{thm}[Theorem~\ref{mthm1}]\label{mthm2}
Let $n, k \in \NN$ such that $n>k\geq\lfloor\frac{n}{2}\rfloor$. 
Then
\begin{equation}\label{eq:sdec}
(K[Z_S],m_S)_{S \in \Sc}
\end{equation}
is a Stanley decomposition of $M(n,k)$, where $Z_S$ is as in \eqref{eq:Zs} and 
$m_S := X^{S \setminus G(S)} \partial(e_{G(S)})$.
\end{thm}

\newcommand{\Gc}{\mathcal{G}}
\newcommand{\Cc}{\mathcal{C}}
\newcommand{\supp}{\mathrm{supp}}
To show that this is really a Stanley decomposition, we use the following criterion:
\begin{prop}[Proposition 2.9, \cite{bku10}]\label{hdtosd}
Let $(K[Z_i],s_i)_{i \in I}$ be a Hilbert decomposition of a module $M$.
For every $i \in I$ choose a homogeneous nonzero element $m_i\in M$ of degree $s_i$.
Then $(K[Z_i],m_i)_{i\in I}$ is a Stanley decomposition of $M$, if for every multidegree $m$ the family 
$$\Cc(m)=\set{m_i \with (K[Z_i]m_i)_m\neq 0}$$ 
is linearly independent over $R$.
\end{prop}
It turns out to be more convenient to consider instead the sets 
\[ \Gc(m) := \set{ G(S) \with (K[Z_S] m_S)_m\neq 0 }. \]
Clearly $\Cc(m)$ and $\Gc(m)$ determine each other. Moreover it is easy to see that $\Cc(m)$ and $\Gc(m)$ only depend on the support of $m$: \[ \supp(m)=\set{i\in [n]\with \text{the }i\text{-th component of }m\text{ is non-zero}}.\] So by abuse of notation we write $\Cc(M):=\Cc(m)$ and $\Gc(M):=\Gc(m)$ if $M=\supp(m)$.

We will check the linear independence with the following condition:
\begin{dfn}
A family $\Gc$ of $k$-sets fulfills the \emph{triangle-condition} ($\triangle$-condition), if there is an order $\prec$ on $\Gc$ such that every $G \in \Gc$ contains a $(k-1)$-subset $T$ (the \emph{distinguished} set) which is not contained in the smaller sets, i.\,e.\,$T\nsubseteq H$ for every $H\prec G$.
\end{dfn}
\begin{lem}
If a family $\Gc$ fulfills the $\triangle$-condition, then the set $\set{\partial (e_G) \with G \in \Gc}$ is  linearly independent.
\end{lem}
\begin{proof}
By the definition of the differential map $\partial$ \eqref{eq:diff}, we have that 
$$\partial(e_G)=\pm e_T+\ldots,$$ where $T\subset G$ is the distinguished subset. Because $\Gc$ satisfies the $\triangle$-condition, the term $e_T$ does not appear in $\partial(e_H)$ for every $H\prec G$. Hence, the restriction of the chain map $\partial$ to $\set{e_G \with G \in \Gc}$ forms an (upper) triangular matrix.
\end{proof}
In particular, if $\Gc(m)$ satisfies the $\triangle$-condition then $\Cc(m)$ is linearly independent. So to prove Theorem~\ref{mthm2}, we have to show that $\Gc(m)$ fulfills the $\triangle$-condition for every multidegree $m$. 

For this we need a more explicit description of the sets $\Gc(m)$. Fix a multidegree $m$.
Looking at the Hilbert decomposition we see that $Z_S$ is the whole polynomial ring if $S$ is not in the image of $\psi$, or one variable is missing, which is the one dropped by $\psi$. 
This means that $(K[Z_S] m_S)_m\neq 0$ if and only if $S$ is not in the image of the restriction of $\psi$ to all subsets of $\supp(m)$.
Overall, the set of all contributing generators for a multidegree $m$ is given by
\[
\Gc(M) =\set{G(S) \with S \in \Sc, S\subseteq M, S \notin \mathrm{Im}\psi|_{\mathcal{P}(M)}}.
\]
Here, $M=\supp(m)$ and $\mathcal{P}(M)$ denotes the power set of $M$.
Recall that $G(S) = \psi^{|S|-k}(S)$ and that $|S|-k$ is even by the definition of $\Sc$.
Hence $\Gc(M)$ consists of the $k$-subsets of $M$, which lie in an even power of $\psi$ restricted to $\mathcal{P}(M)$, but not in the odd one:
$$\Gc(M)= \left\{G\in\binom{M}{k} \;\middle|\;\exists\, i\; \exists M'\subseteq M:\psi^{2i}(M')=G\text{ and }\nexists\, M''\subseteq M:\psi^{2i+1}(M'')=G \right\}.$$
Consequently, the following definition is quite useful:
\begin{dfn}
For a subset $G\subseteq M$ the \emph{index} of $G$ in $M$ is defined as 
$$\ind_M (G)=\max\set{i \with \exists\, M'\subseteq M: \psi^{i}(M')=G }.$$
\end{dfn}
This allows us to write
$$\Gc(M)= \left\{G\in\binom{M}{k} \;\middle|\; \ind_M(G) \text{ is even} \right\}.$$

Note that by Proposition~\ref{psiphi} the index can also be expressed in terms of $\phi$:
$$\ind_M(G)=\max\{i\mid \phi^i(G) \text{ is defined and }\phi^i(G)\subseteq M \}.$$ This formula can be quite useful for later computations and proofs. 
\begin{exmp}
For the multidegree $12457$ we have one element of index 2 (147) and five elements of index 0 (125, 127, 245, 257), as it can be seen in Figure~\ref{fig:ex}.
 \begin{figure}[h]
 \centering
 \includegraphics{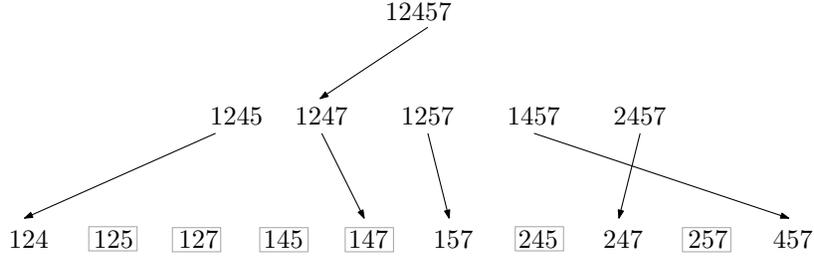}
 \caption{Subsets of $12457$ with even index.}
 \label{fig:ex}
 \end{figure}
\end{exmp}

As a final step, we show that the \emph{squashed} order works for the family $\Gc(M)$. It is defined for two sets of the same size as $$G\succ H\,:\Leftrightarrow\,\max G\Delta H\in G,$$ where $\Delta$ denotes the symmetric difference. For details, see \cite[Ch. 7]{and87}.

Moreover, the distinguished subsets are the $\widetilde{\psi}(G)$'s where the function is defined as before, but drops the restriction of Section~\ref{mba}:
\begin{align*} \widetilde{\alpha}(G)&:=\max_{g\in G}\{\rho_G(g)\},\\
\widetilde{N}(G)&:=\{g\in G\mid \rho_G(g)=\widetilde{\alpha}(G)\},\quad \widetilde{\nu}(G):=\min\widetilde{N}(G),\\
\widetilde{\psi}(G)&:=G\setminus\{{\widetilde{\nu}(G)}\}.
\end{align*} So the function is always defined, but we lose the injectivity. As we will see later, this does not impose a problem for our result.

We continue by showing that in most cases if a set succeeds another set in the squashed order and the latter set contains the claimed distinguished subset of the former, the index increases exactly by one.

Recall that $$\rho_G(g)=\sum_{j=0}^{g}\chi_G(j),\quad g\in\{0,1,\dots,n\}.$$
\begin{lem}\label{lem:ind}
Let $G$ and $H$ be subsets of a set $M\subseteq\{1,\dots,n\}$ with $|G|=|H|\geq\lfloor\frac{n}{2}\rfloor$. Furthermore let $G\succ H$ and $\wpsi(G)\subset H$. Then the following hold:
\begin{enumerate}
\item If $\walpha(G)\geq0$ then $\ind_M(H)=\ind_M(G)+1$;
\item If $\walpha(G)<0$ then $\ind_M(H)=1$.
\end{enumerate}
\end{lem}
\begin{proof}
By assumption we have $$\max G\Delta H=\max\{h,\widetilde{\nu}(G)\}=\wnu(G),\quad \wnu(G)\in G,h\in H.$$
The definition of $\rho$ implies the following equation:
\begin{equation}\label{eq:gh}
\rho_{H}(g)=\begin{cases}
\rho_G(g) & g <h,\\
\rho_G(g)+2 & h\leq g<\wnu(G),\\
\rho_G(g) & g\geq\wnu(G).
\end{cases}
\end{equation}
Since $\rho_{G}(g)<\walpha(G)$ for $g<\wnu(G)$ it is $\rho_{H}(g)\leq\walpha(G)+1$ for $g<\wnu(G)$. Furthermore $\rho_{H}(g)=\rho_G(g)\leq\walpha(G)$ for all $g\geq\wnu(G)$. So overall $\walpha(H)\leq\walpha(G)+1$.

 On the other hand 
 \begin{equation}\label{eq:alpha_h}
\walpha(H)\geq\rho_{H}(\wnu(G)-1)=\rho_G(\wnu(G))+1>\rho_G(\wnu(G))=\walpha(G),
 \end{equation} 
by \eqref{eq:gh}. This shows that $\walpha(H)=\walpha(G)+1$. Furthermore \eqref{eq:gh} and \eqref{eq:alpha_h} yield that 
\begin{equation}\label{eq:rho_h}
\rho_{H}(\wnu(G)-1)=\widetilde{\alpha}(G)+1>\rho_{G}(g)=\rho_{H}(g),\quad\forall\,g\geq\wnu(G).\end{equation}
Next we show that $\walpha(H)\geq0$. This is obvious if $\walpha(G)\geq0$ by \eqref{eq:alpha_h}. If $\walpha(G)< 0$ then $|G|\geq\lfloor\frac{n}{2}\rfloor$ implies $|G|=|H|=\lfloor\frac{n}{2}\rfloor$. Note that this can only happen if $n$ is odd. Furthermore this implies that $$\rho_G(n)=\rho_H(n)=2|G|-n=-1$$ and hence $\walpha(G)=-1$. So in both cases $\walpha(H)\geq 0$ and thus \eqref{eq:rho_h} shows that:  
\begin{align}\label{eq:mu_h}
\mu(H)&=\wnu(G)-1,
\intertext{and thus} \label{eq:phi_h}
 \phi(H)&=H\cup\{\wnu(G)\}.
\end{align}
\begin{asparaenum}
\item Assume that $\walpha(G)\geq 0$ and thus $\walpha(H)>0$. We know by the definition of $\phi$ and \eqref{eq:gh} that 
\begin{equation}\label{eq:parallel}
\rho_{\phi(H)}(g)=\rho_{G}(g)+2,\quad\forall g\geq h.
\end{equation}
Since $\walpha(G)\geq0$ we have that $\mu(G)>0$ and thus $\mu(G)\geq \wnu(G)>h$ and $\mu(\phi(H))>\mu(H)=\wnu(G)-1\geq h$.
 Hence \eqref{eq:parallel} implies that $\mu(\phi(H))=\mu(G)$, as illustrated in Figure~\ref{fig:proof}. 

Note that \eqref{eq:parallel} stays valid if $\phi$ is applied to both $G$ and $\phi(H)$. Moreover $\mu(\phi^2(H))>\mu(\phi(H))$ and $\mu(\phi(G))>\mu(G)$. Hence $\mu(\phi^2(H))=\mu(\phi(G))$.
Continuing this argument we see that $\phi^i(G)$ is defined and $\phi^i(G)\subseteq M$ if and only if $\phi^i(\phi(H))$ is defined and is a subset of $M$. This shows that $$\ind_M(G)=\ind_M(\phi(H))=\ind_M(H)-1.$$
\begin{figure}[h]
\centering
\includegraphics{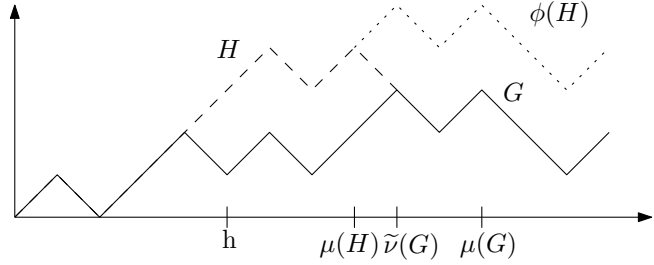}
\caption{Illustration of Lemma~\ref{lem:ind}}
\label{fig:proof}
\end{figure}
\item Since \eqref{eq:mu_h} holds in the case $\walpha(G)<0$ as well, we know that $\ind_M(H)\geq1$. Moreover $$\rho_{\phi(H)}(n)=1=\alpha(\phi(H)),$$ so $\mu(\phi(H))=n$ and thus $\ind_M(H)=1$.
\end{asparaenum}

\end{proof}
\begin{rem}
Note that in the case $\walpha(G)<0$ the index of $G$ in $M$ is $1$ or $0$ depending whether $1\in M$ or not. Hence the case distinction is necessary. 
\end{rem}
Now we can show that the squashed order works:
\begin{prop}\label{sc}
The family $\Gc(m)$ fulfills the $\triangle$-condition with respect to the squashed order for every multidegree $m$.
\end{prop}
\begin{proof}
Let $M=\supp(m)$. The set $\Gc(M)$ contains only $k$-subsets with even index.
Hence by Lemma~\ref{lem:ind}, $G\succ H$ implies that $\wpsi(G)\nsubseteq H$.
\end{proof}

\section*{Acknowledgment}
 The authors wish to thank Professor Bruns for suggesting and discussing the topic of this paper.

\bibliography{lit}
\bibliographystyle{alpha}

\end{document}